\documentclass[a4paper, 11pt]{amsart}

\usepackage{amsmath, amsthm, amssymb}
\usepackage{mathrsfs}
\usepackage[Symbol]{upgreek}
\usepackage{amscd,epsfig,esint}
\usepackage{verbatim}
\usepackage[english]{babel}

\usepackage{pict2e}
\usepackage[]{color}

\theoremstyle{plain}
\newtheorem{theorem}{Theorem}
\newtheorem{lemma}[theorem]{Lemma}
\newtheorem{corollary}[theorem]{Corollary}
\newtheorem{proposition}[theorem]{Proposition}
\theoremstyle{definition}
\newtheorem{definition}[theorem]{Definition}

\theoremstyle{remark}

\DeclareMathOperator{\esup}{ess\,sup}

\newcommand{\R}{\mathbb{R}}
\newcommand{\RN}{\mathbb{R}^n}
\newcommand{\RH}{\hat{\mathbb{R}}^n}

\newcommand{\N}{\mathbb{N}}
\newcommand{\Leb}{\mathcal{L}}

\newcommand{\M}{\mathcal{M}}
\newcommand{\supp}{\mathrm{supp}}
\newcommand{\B}{\mathcal{B}}

\newcommand{\Langle}{\big\langle}
\newcommand{\Rangle}{\big\rangle}

\newcommand{\di}{\partial}

\newcommand{\diver}{\mathrm{div}}

\begin{document}
\thispagestyle{empty}

\title[Well-posedness for continuity equation]{Well-posedness for the continuity equation for vector fields with suitable modulus of continuity}
\author{Albert Clop}
\author{Heikki Jylh\"a}
\author{Joan Mateu}
\author{Joan Orobitg}
\date{\today}

\begin{abstract}
We prove well-posedness of linear scalar conservation laws using only assumptions on the growth and the modulus of continuity of the velocity field, but not on its divergence. As an application, we obtain uniqueness of solutions in the atomic Hardy space, $H^1$, for the scalar conservation law induced by a class of vector fields whose divergence is an unbounded $BMO$ function.
\end{abstract}

\maketitle

\section{Introduction}

\noindent
The scalar continuity equation
\begin{equation}\label{ce}
\begin{cases}
\frac{d}{dt}\, \rho + \diver(b\,\rho)=0\\ \rho(0,\cdot)= \rho_0
\end{cases}
\end{equation} 
appears in many conservation phenomena in nature. Among the most representative ones, we mention the so called aggregation equation (see for instance \cite{BLL}), which describes the time evolution of an active scalar $\rho(t,\cdot)$ with initial state $\rho_0$, under a velocity field
$$ b(t,\cdot)=K\ast \rho(t,\cdot).$$
The dependence of $b$ with respect to the unknown $u(t,\cdot)$ makes the aggregation equation nonlinear, while the analytic nature of the kernel $K=\nabla N$ ($N$ is the Newtonian potential) gives it a gradient flow structure.  In the present paper, though, we omit both nonlinearities and gradient flow structure, and focus our attention on much simpler linear case. \\
\\
It is known by the so-called superposition principle (see for instance \cite{AC}) that if a vector field $b$ admits a unique flow $\phi$ of trajectories $t\mapsto\phi_t(x)$ then non-negative solutions to \eqref{ce} are uniquely determined by $\rho_0$. For this reason, in the present paper we focus on the much more subtle question of signed solutions. \\
\\
Having a mass conservation structure, one may think that continuity equations have $L^1$ as the natural function space for solutions. It is then reasonable to think its adjoint equation, also known as \emph{transport equation},
$$
\begin{cases}
\frac{d}{dt}\, \omega + b\cdot\nabla\omega=0,\\
\omega(0,\cdot)=\omega_0.
\end{cases}
$$
in the $L^\infty$ setting. Indeed, the classical theory by DiPerna and Lions \cite{DPL}, together with Ambrosio's more recent developments \cite{A}, provides an existence, uniqueness and stability theory for solutions $u\in L^\infty(0,T;L^\infty)$ for any Sobolev vector field $b$ which, among other assumptions, satisfies $\diver(b)\in L^\infty$. There is, though, an increasing interest in removing both boundedness assumptions on the divergence \cite{CC, CJMO, CJMO2, CCS, Mu} as well as on the datum \cite{CJMO3}. In this setting, the Euler system of equations deserves a special mention. Its scalar formulation in the plane,
\begin{equation}\label{vorticity}
\frac{d\omega}{dt} + v\cdot \nabla\omega=0,\hspace{1cm}\text{where }v(t,\cdot)=\frac{i}{2\pi\overline{z}}\ast\omega(t,\cdot),
\end{equation}
has the structure of a nonlinear transport equation, together with the initial condition $v(0,\cdot)=\frac{i}{2\pi \overline{z}}\ast \omega(0,\cdot)$. After Yudovich's proof of global existence and uniqueness of solutions $\omega\in L^\infty(0,T;L^\infty)$ for $\omega_0\in L^1\cap L^\infty$ \cite{Y}, there has been many attempts to understand the case of unbounded vorticities. Particular attention is devoted to spaces that stay close to $BMO$, the space of functions of \emph{bounded mean oscillation}, which arises naturally since it contains the image of $L^\infty$ under any Calder\'on-Zygmund singular integral operator. It is conjectured that well-posedness of \eqref{vorticity} fails in $BMO$.  Also, strong ill-posedness has been proven in certain Sobolev spaces included in $BMO$ \cite{BLi}. However, there exist classes of unbounded vorticities for which \eqref{vorticity} is well-posed \cite{BH,BK}. In a similar way, the present paper deals with the well-posedness of  \eqref{ce} in the atomic Hardy space $H^1$, which is the predual of $BMO$, and which consists of $L^1$ functions having $L^1$ image under any Calder\'on-Zygmund operator. \\
\\
Cauchy problems for the linear continuity equation with non-smooth velocity fields were successfully treated with the DiPerna-Lions scheme and the notion of renormalized solution, as well as the more recent extensions by Ambrosio in the $BV$ setting. In both approaches, the starting point is the classical Cauchy-Lipschitz theory, which allows to write the solution $\rho=\rho(t,x)$ of 
\begin{equation}\label{linearce}
\begin{cases}
\frac{d}{dt}\,\rho +\diver(b\, \rho) =0\\
\rho(0,\cdot)=\rho_0
\end{cases}
\end{equation}
as the adjoint composition operator,
\begin{equation}\label{transport}
\rho(t,x)=\rho_0\circ\phi_t^{-1}(x)\,J(x,\phi_t^{-1})
\end{equation}
where $\phi_t:\R^n\to\R^n$ is the flow generated by the velocity field $b$,
\begin{equation}\label{flowODE}
\begin{cases}
\dot{\phi_t}(x)=b(t, \phi_t(x)),\\\phi_0(x)=x,
\end{cases}
\end{equation}
and $J(x,\phi_t^{-1})$ denotes the jacobian determinant of the inverse flow $\phi_t^{-1}$ at the point $x$, at least for smooth enough  $b$. Abusing of the classical mass transport notation, one could equivalently write $\rho(t,\cdot)=(\phi_t)_\sharp \rho_0$. Towards finding explicit solutions $\rho\in L^\infty(0,T; H^1(\R^n))$ of the problem \eqref{linearce} for a given $\rho_0\in H^1(\R^n)$, there are two things to be analyzed. First, describing the class $\mathcal{Q}$  of diffeomorphisms $\phi_t$ under which \eqref{transport} defines a bounded operator in $H^1(\R^n)$. Second, describing the class of velocity fields $b$ such that \eqref{flowODE} has a solution $\phi_t$ that falls into $\mathcal{Q}$. Both questions drive us naturally to Reimann, who partially solved them both in two papers in the 70's \cite{Re, Re2}. In the first one, quasiconformality was found to be the fundamental notion. In the second, uniform bounds for the anticonformal part of $Db$ were proven to be enough. Let us recall that the anticonformal part of $Db$ is, by definition,
$$S_Ab(t,x)=\frac{Db(t,x)+D^Tb(t,x)}{2} -\frac{\diver b(t,x)}{n}\,I_n,$$
where $I_n$ denotes the $n$-dimensional identity matrix.

\begin{theorem}\label{mainH1}
Let $b:[0,T)\times\R^n\to\R^n$ be a vector field, $b\in L^1(0,T; W^{1,1}_{loc}(\R^n))$, such that $S_Ab\in L^1(0,T; L^\infty(\R^n))$. Assume also that
$$
\frac{|b(t,x)|}{1+|x|\,\log(e+|x|)}\in L^1(0,T; L^\infty(\R^n)).
$$ 
For every $h_0\in H^1(\R^n)$, the problem
\begin{equation}\label{H1problem}
\left\{ 
\begin{array}{rll}
\di_t h +\diver(b h) & =0 & \textrm{in}\; (0,T)\times\RN \\
h(0,\cdot) & =h_0 & \textrm{in}\; \RN.
\end{array}\right.
\end{equation}
admits a unique weak solution $h\in L^\infty(0,T; H^1(\R^n))$. 
\end{theorem}

\noindent
The essential point for existence is the fact that quasiconformal maps transport boundedly measures with $H^1(\R^n)$ density into themselves. For this, we rely on previous results by Reimann \cite{Re,CJMO3} and 
Fefferman's $H^1-BMO$ duality. \\
\\
In the proof of uniqueness in Theorem \ref{mainH1}, most of the available literature cannot be used, due to the unboundedness of $\diver(b)$. However, we remark two exceptions. First, a work by  Ambrosio - Bernard \cite{AB}, where uniqueness is obtained from diagonal modulus of continuity assumptions on $b$, together with global boundedness. Unfortunately, this was not enough for proving Theorem \ref{mainH1}, which includes many unbounded vector fields. We found a second exception in a paper by Seis \cite{Seis}, which provides a detailed stability estimate for solutions $u\in L^1(0,T;L^q(\R^n))$ of continuity equations with a velocity field $b\in W^{1,p}_{loc}(\R^n;\R^n)$, $\frac1p+\frac1q=1$, $p\in(1,\infty]$. It turns out that a modification of the argument in \cite{Seis} brings the following uniqueness result for solutions in $L^1(0,T; L^1(\R^n))$ (actually a more general one, see Section \ref{proofmainPROP} for details).

\begin{proposition}\label{mainPROP}
Let $b:[0,T]\times\RN\to\RN$ be a vector field. Assume that 
\[\label{a1}
|b(t,x)|\leq C(1+|x|)   \tag*{(A1)} ,
\]
and that 
\[\label{a2}
|b(t,x)-b(t,y)|\leq C\,|x-y|\,\log(e+1/|x-y|), \tag*{(A2)}  
\]
of all $x,y\in\RN$ and almost every $t\in [0,T]$.  If $\rho_0\in L^1(\RN)$, then the Cauchy problem
$$
\left\{ 
\begin{array}{rll}
\di_t \rho +\diver(b \rho) & =0 & \textrm{in}\; (0,T)\times\RN \\
\rho(0,\cdot) & =\rho_0 & \textrm{in}\; \RN 
\end{array}\right.
$$
has at most one weak solution $\rho\in L^1(0,T;L^1(\RN))$. 
\end{proposition}

\noindent
The above result still does not provide a complete proof for uniqueness in Theorem \ref{mainH1}. Indeed, there are vector fields $b$ admissible for Theorem \ref{mainH1} for which \ref{a1} and \ref{a2} may both fail at the same time. It is worth mentioning too that existence of solutions may fail in the setting of Proposition \ref{mainPROP}, since flows arising from the above velocity fields $b$ do not preserve Lebesgue measurable sets in general. See \cite{CC, CJMO, CJMO2} for optimal conditions in this direction. In particular, these flows could move an initial datum $\rho_0$ away from $L^1(\RN)$ and transform it into a non absolutely continuous measure. It is then natural to look at the case where $\rho(t,\cdot)$ belongs to the space of signed Borel measures on $\RN$ with finite total variation, which is denoted by $\M(\RN)$. It turns out that this is the right choice for a full well-posedness theorem.

\begin{theorem}\label{mainTHM}
Let $b:[0,T]\times\RN\to\RN$ be a vector field such that the following two conditions hold: There exists a continuous and nondecreasing function $G:[0,\infty)\to(0,\infty)$ satisfying $\int_r^\infty \frac{ds}{G(s)}=\infty$ for some $r>0$ such that
\[\label{b1}
\sup_{x\in\RN} \frac{|b(t,x)|}{G(|x|)} \in L^\infty(0,T) \tag*{(B1)},
\]
and there exists a continuous and nondecreasing function $\omega:[0,\infty)\to[0,\infty)$ satisfying $\int_0^r \frac{ds}{\omega(s)} = \infty$ for some $r>0$ such that
\[\label{b2}
\sup_{x,y \in B(0,R)} \frac{|b(t,x)-b(t,y)|}{\omega(|x-y|)} \in L^\infty(0,T). \tag*{(B2)}
\]
for any radius $R>0$. Then for any $\rho_0\in\M(\RN)$ there exists a unique solution (as defined in Definition \ref{CPsol}) $\rho\in L^1(0,T;\M(\RN))$ for the Cauchy problem
\begin{equation}\label{CPb}
\left\{ 
\begin{array}{rll}
\di_t \rho +\diver(b \rho) & =0 & \textrm{in}\; (0,T)\times\RN \\
\rho(0,\cdot) & =\rho_0 & \textrm{in}\; \RN.
\end{array}\right.
\end{equation}
Moreover, this solution is of the form $\rho(t,\cdot)=(\phi_t)_\sharp\rho_0$, where $\phi_t$ is the flow of the vector field $b$.
\end{theorem}

\noindent
It is worth mentioning that Theorem \ref{mainTHM} holds under slightly more generality. Namely, if one further knows that $\rho\in L^\infty(0,T;\M(\RN))$, then uniqueness still holds if one replaces $L^\infty(0,T)$ by $L^1(0,T)$ in both \ref{b1} and \ref{b2}. The argument for obtaining existence in Theorem \ref{mainTHM} is straightforward, because of to the measure-valued setting and the continuity of $b$. Concerning uniqueness, the proof is much trickier. Exactly as in Proposition \ref{mainPROP}, the proof expands even further the methods used by Seis in \cite{Seis} using optimal transport. Among the consequences, uniqueness in Theorem \ref{mainH1} now follows from Theorem \ref{mainTHM}. Moreover,  it also provides uniqueness of solutions in $L^1(0,T; L^1(\RN))$ for the continuity equation in these cases where one can guarantee the preservation of Lebesgue measurable sets under the flow. This is the case for instance if
$$
\diver(b)\in \operatorname{Exp}\left(\frac{L}{\log L \log\log L \cdots}\right),
$$
see also \cite{CJMO2}. Nevertheless, conditions \ref{b1}, \ref{b2} are neither a consequence of, nor a reason for any condition on the divergence of $b$. Some counterexamples in this direction are given at the end of Section \ref{proofmainTHM}.  \\
\\
The paper is structured as follows. In Section \ref{prelim} we give some definitions and basic results for continuity equations. In Section \ref{masstransprelim} we overview some results from optimal transport theory. In Section \ref{proofmainPROP} we prove a slightly more general version of Proposition \ref{mainPROP}, which allows us to present the ideas behind the proof of Theorem \ref{mainTHM} without getting hung up on all the details. In Section \ref{proofmainTHM} we prove Theorem \ref{mainTHM} and give some counterexamples for the optimality of its assumptions. In Section \ref{applications} we prove Theorem \ref{mainH1}. All along the paper we will be abusing notation, so that we identify $\rho(t,\cdot)=\rho_t$ and $\phi(t,\cdot)=\phi_t$.\\
\\
{\textbf{Acknowledgements}}. A.C., J.M. and J. O. were partially supported by research grants 2014SGR75 (Generalitat de Catalunya) and MTM2016-75390-P (spanish government). All named authors were partially supported by the research grant FP7-607647 (European Union).

\section{Basic properties of continuity equations}\label{prelim}

\noindent
Let us first introduce our notation. By $\M_+(\RN)$ we denote the set of finite nonnegative Borel measures on $\RN$. 
We denote by $\M(\RN)$ the space of signed Borel measures on $\RN$ with finite total variation, that is, $\mu\in\M(\RN)$ if $\mu$ is a signed Borel measure and
\[
\|\mu\|_{TV}:=\sup \Big\{ \sum_{j\in\N}|\mu(A_j)| : A_j\in\B(\RN), \, A_i\cap A_j = \emptyset \;\textrm{for}\, i\ne j \Big\} <\infty,
\]
where $\B(\RN)$ is the Borel $\sigma$-algebra. For $\mu\in\M(\RN)$ we denote by $|\mu|$ the total variation measure of $\mu$. We use the Jordan decomposition of measures: Any $\mu\in\M(\RN)$ can be written as $\mu=\mu^+-\mu^-$ for some mutually singular measures $\mu^+,\mu^-\in\M_+(\RN)$ 
Given $\mu\in\M(\RN)$ and a homeomorphism $\phi:\RN\to\RN$ we define the \emph{push-forward} $\phi_\sharp\mu$ by 
\[
\phi_\sharp\mu(A) := \mu(\phi^{-1}(A)) \quad\textrm{for any}\; A\in\B(\RN).
\]
Equivalently, one can define $\phi_\sharp\mu$ by duality with the space $C_0(\R^n)$ of continuous functions vanishing at infinity,
$$\langle \phi_\sharp\mu, \varphi\rangle:=\int \varphi(\phi(x))\,d\mu(x),\hspace{.5cm}\forall\varphi\in C_0(\R^n).$$  
Then it is easy to see that $\phi_\sharp\mu\in\M(\RN)$ with $\|\phi_\sharp\mu\|_{TV} \le \|\mu\|_{TV}$.

\begin{definition}\label{CPsol}
A measure valued map  $\rho\in L^1(0,T; \M(\RN))$ is a \emph{weak solution} of the Cauchy problem \eqref{CPb} for the continuity equation 
if the equality
\[
\int_0^T \int_{\RN} \big[ \di_t\varphi(t,x)+\langle b(t,x), \nabla \varphi(t,x) \rangle \big] d\rho_t(x)\, dt + \int_{\RN} \varphi(0,x)\,d\rho_0(x) =0
\]
holds for all functions $\varphi\in C_c^\infty([0,T)\times \RN)$.
%
\end{definition}

\noindent
Let's point out some simple properties of solutions, which we will prove for completeness. First is related to the weak continuity of solutions. 

\begin{lemma}\label{wcont}
Let $b\in L^\infty(0,T;L_{loc}^\infty(\RN))$. Let $\rho\in L^1(0,T;\M(\RN))$ be a solution for the Cauchy problem \eqref{CPb}.
Then we have the convergence 
\[
\int_{\RN} \varphi(x)\,d\rho_t(x) \stackrel{t\to 0}{\longrightarrow} \int_{\RN} \varphi(x)\,d\rho_0(x)
\]
for every $\varphi\in C_c^\infty(\RN)$.
\end{lemma}

\begin{proof}
Let $\varphi\in C_c^\infty(\RN)$. Fix $t_0\in(0,T)$ and $\varepsilon\in(0,t_0)$. Choose a nonnegative function $\eta_{t_0,\varepsilon}\in C_c^\infty(t_0-\varepsilon,t_0+\varepsilon)$ such that $\int \eta_{t_0,\varepsilon}(s)\,ds=1$, and define $\psi_{t_0,\varepsilon}(t):= 1-\int_0^t\eta_{t_0,\varepsilon}(s)\,ds$. 
Then using $\psi_{t_0,\varepsilon}\varphi$ as a test function we get
\begin{align*}
& \int_0^T \psi_{t_0,\varepsilon}'(t)\int_{\RN} \varphi\,d\rho_t\,dt +\int_0^T \psi_{t_0,\varepsilon}(t)\int_{\RN} \Langle b(t,x) , \nabla\varphi(x) \Rangle d\rho_t(x)\,dt \\
& \qquad \qquad + \int_{\RN} \varphi\,d\rho_0 \quad = \quad 0.
\end{align*}
Taking $\varepsilon\to 0$ we obtain (for a.e.~$t_0$, if we are precise)
\begin{equation}\label{t0eq}
-\int_{\RN} \varphi \,d\rho_{t_0} +\int_0^{t_0} \int_{\RN} \Langle b(t,x) , \nabla\varphi(x) \Rangle d\rho_t(x)\,dt + \int_{\RN} \varphi \,d\rho_0 = 0.
\end{equation}
Thus we can deduce that
\[
\Big| \int_{\RN} \varphi \,d\rho_{t_0}-\int_{\RN} \varphi \,d\rho_0 \Big| \le \Big| \int_0^{t_0} \int_{\RN} \Langle b(t,x) , \nabla\varphi(x) \Rangle d\rho_t(x)\,dt \Big| \to 0,
\]
as $t_0\to 0$, which is what we wanted to prove.
\end{proof}

\noindent
Before the next Lemma we need to introduce speciffic cut-off functions that will be used throughout the paper. Let $k\in\N$ and consider the function
\[
\tilde{\chi}_k(x) = \left\{ \begin{array}{ll} 
1, & \textrm{if}\; |x|\le k \\
\max\{0, 1-\int_k^{|x|} \frac{dr}{G(r)} \}, & \textrm{if}\; |x|> k.
\end{array} \right.
\]
Modifying this function we find a smooth, compactly supported function $\chi_k\in C_c^\infty(\RN)$ such that 
\begin{equation}\label{cutoff}
\begin{split}
& \chi_k=1 \;\;\; \textrm{in}\;\; B(0,k), \\ 
& \chi_k=0 \;\;\;\textrm{in}\;\; \RN\setminus B(0,R_k) \;\;\;\textrm{for some} \;\; k<R_k<\infty \\ 
\textrm{and} \quad & |\nabla \chi_k(x)|\le \frac{2}{G(|x|)}.
\end{split}
\end{equation}
Here the existence of $R_k$ follows from the assumption $\int_r^\infty \frac{ds}{G(s)} =\infty$ for some (and thus any) $r>0$. \\
\\
The second important property of solutions is the conservation of mass, or from the point of view of this paper, the conservation of mass balance.

\begin{lemma}\label{0int}
Suppose the vector field $b$ satisfies the condition \ref{b1} and that $\rho\in L^1(0,T;\M(\RN))$ is a weak  solution for the Cauchy problem
\begin{equation}\label{0eq}
\left\{ 
\begin{array}{rl}
\di_t \rho +\diver(b \rho) & = 0 \\
\rho(0,\cdot) & = 0 
\end{array}\right.
\end{equation}
Then for a.e.~$t\in[0,T]$ it holds that $\rho_t(\RN)=\int_{\RN}d\rho_t=0$, or in other words: $\rho_t^+(\RN)=\rho_t^-(\RN)$.
\end{lemma}

\begin{proof}
Use a test function $\varphi(t,x)=\psi(t)\chi_k(x)$, where $\psi\in C_c^\infty(0,T)$. This gives
\[
\int_0^T \psi'(t)\int_{\RN} \chi_k(x)\,d\rho_t(x)\,dt = - \int_0^T \psi(t) \int_{\RN} \big\langle b(t,x), \nabla\chi_k(x) \big\rangle d\rho_t(x).
\]
Since this works for every $\psi\in C_c^\infty(0,T)$, we obtain
\[
\di_t\Big( \int_{\RN}\chi_k(x)\,d\rho_t(x)\Big)=\int_{\RN} \Langle b(t,x), \nabla\chi_k(x) \Rangle d\rho_t(x).
\]
Combining this with Lemma \ref{wcont} we get
\begin{align*}
& |\rho_t(\RN)| =\lim_{k\to\infty} \Big| \int_{\RN}\chi_k(x)\,d\rho_t(x) \Big| \\
\le & \limsup_{k\to\infty}  \int_0^t \Big| \int_{\RN} \Langle b(s,x), \nabla\chi_k(x) \Rangle d\rho_s(x) \Big|ds \\
\le & 2 \Big\| \frac{|b(t,x)|}{G(|x|)} \Big\|_{L^\infty((0,T)\times \RN)} \limsup_{k\to\infty}  \int_0^t \int_{\RN\setminus B(0,k)} d|\rho_s|(x) ds =0,
\end{align*}
and thus conclude the proof.
\end{proof}

\noindent
In the smooth case, the solution to the continuity equation can be found using the method of characteristics. Actually, this works also in our non-smooth case.

\begin{proof}[Proof of the existence part of Theorem \ref{mainTHM}]
Let $\phi$ be the flow of $b$. This flow exists by the classical Cauchy-Lipschitz theory, and moreover it satisfies
\begin{equation}\label{DE}
\phi(t,x)=x+\int_0^t b\big(s,\phi(s,x)\big)\,ds \quad\textrm{for any}\; x\in\RN.
\end{equation}  
We need to show that $\rho_t:=\phi(t,\cdot)_\sharp \rho_0$ is a solution to the Cauchy problem \eqref{CPb}. First, the properties of the push-forward ensure that we actually have $\rho_t\in L^\infty(0,T;\M(\RN))$, since $\|\rho_t\|_{TV}$ is bounded by $\|\rho_0\|_{TV}$. 

Let $\varphi\in C_c^\infty([0,T)\times\RN)$. First notice that $\varphi$ is Lipschitz and $t \mapsto\phi(t,x)$ is absolutely continuous for any $x\in\RN$. This implies that the function $f_x(t):=\varphi(t,\phi(t,x))$ is absolutely continuous for any $x\in\RN$ and 
\begin{equation}\label{ftc}
\int_0^T f'_x(t)\,dt =f_x(T)-f_x(0)=-\varphi(0,x).
\end{equation}
On the other hand if we fix $x\in\RN$, then for $\Leb^1$-a.e.~$t\in[0,T)$ we have
\begin{equation}\label{derivative}
\begin{split}
f'_x(t) & =\di_t\varphi(t,\phi(t,x))+\big\langle \nabla\varphi(t,\phi(t,x)), \di_t\phi(t,x) \big\rangle \\
& = \di_t\varphi(t,\phi(t,x))+\big\langle b(t,\phi(t,x)), \nabla\varphi(t,\phi(t,x)) \big\rangle.
\end{split}
\end{equation}
Now we can check that $\rho_t$ is a solution. Using the definition of the push-forward measure and applying Fubini's Theorem, \eqref{derivative} and \eqref{ftc} we calculate
\begin{align*}
& \int_0^T\int_{\RN} \big[ \di_t\varphi(t,x)+ \big\langle b(t,x), \nabla\varphi(t,x) \big\rangle \big] d\rho_t(x)\,dt +\int_{\RN}\varphi(0,x)\,d\rho_0(x) \\
= &  \int_0^T\int_{\RN} \big[ \di_t\varphi(t,\phi(t,x))+\big\langle b(t,\phi(t,x)), \nabla\varphi(t,\phi(t,x)) \big\rangle \big] d\rho_0(x)\,dt +\int_{\RN}\varphi(0,x)\,d\rho_0(x) \\
= & \int_{\RN}\int_0^T f'_x(t) \,dt\,d\rho_0(x) +\int_{\RN}\varphi(0,x)\,d\rho_0(x) =0.
\end{align*}
Since $\varphi\in C_c^\infty([0,T)\times \RN)$ was arbitrary, $\phi(t,\cdot)_\sharp\rho_0$ is a solution to the Cauchy problem \eqref{CPb}.
\end{proof}

\section{Useful results from the theory of optimal transport}\label{masstransprelim}

\noindent
In this section we present the part of the theory of optimal transport which we will use in the uniqueness proof. For further interest to this topic we refer to the books \cite{Villani1}, \cite{Villani2} by Villani, as well as \cite{Santambrogio} by Santambrogio and \cite{AGS} by Ambrosio, Gigli and Savar\'e.\\
\\
We want to transport mass in $\RN$ but our use of cut-off functions might lead to a difference between the initial and final mass. This could be a problem. The typical way (see e.g.~\cite{Ekeland}, \cite{CaffMcC} and \cite{HJ}) to avoid this mass inbalance is to add an isolated point $\Diamond$ to $\RN$. We write $\RH:=\RN\cup \{\Diamond\}$. We also ''extend'' the euclidean distance to $\RH$ by setting $|x-\Diamond|=|\Diamond-x|=\infty$ whenever $x\in\RN$ and obviously $|\Diamond-\Diamond|=0$.\\
\\
Given two Borel measures $\mu,\nu\in\M_+(\RH)$ with $\mu(\RH)=\nu(\RH)$ we denote the set of \emph{transport plans} between $\mu$ and $\nu$ by $\Pi(\mu,\nu)$. In other words, a measure $\lambda\in\M_+(\RH\times\RH)$ belongs to $\Pi(\mu,\nu)$, if 
\[
\lambda(A\times\RH)=\mu(A) \quad\textrm{and}\quad \lambda(\RH\times A)=\nu(A)
\]
for all Borel sets $A\subset\RH$. Equivalently, $\lambda\in\Pi(\mu,\nu)$ if
\[
\int_{\RH\times\RH} \big(u(x)+v(y) \big)d\lambda(x,y) =\int_{\RH}u(x)\,d\mu(x) + \int_{\RH} v(y)\,d\nu(y)
\]
for all $u\in L^1(\mu)$, $v\in L^1(\nu)$. If in addition to the measures $\mu,\nu\in\M_+(\RH)$ we are given a continuous cost function $c:\RH\times\RH\to[0,\infty]$, we can study the \emph{optimal transport problem}
\[
\inf_{\lambda\in\Pi(\mu,\nu)} \int_{\RH\times\RH} c(x,y)\,d\lambda(x,y).
\] 
The existence of minimizers, called \emph{optimal transport plans}, can be proved with the direct method in the calculus of variations (see e.g. \cite[Theorem 4.1]{Villani2}). In our case, we are interested in cost functions of the form $c(x,y)=c(|x-y|)$. With our notation, this means that $c(x,\Diamond)=c(\Diamond,x)=c(\infty)$ is a constant among all $x\in\RN$. We assume that $c:[0,\infty]\to[0,\infty]$ satisfies the following conditions:
\begin{equation}\label{cost}
\left\{
\begin{array}{l}
c(0)=0, \;\; c(s)>0 \;\;\textrm{for every $s>0$} \;\;\textrm{and $c$ is nondecreasing} \\
\textrm{$c$ is bounded and continuous} \\
\textrm{$c$ is concave} \\ 
\textrm{$c$ is Lipschitz w.r.t.~the euclidean distance in $[0,\infty)$}.
\end{array}\right.
\end{equation}
We set
\[
C_c(\mu,\nu)  :=\min_{\lambda\in\Pi(\mu,\nu)} \int_{\RH\times\RH} c(|x-y|)\,d\lambda(x,y)
\]
We will study the \emph{total transport cost} $C_c(\mu,\nu)$ for a certain family of cost functions and then compare the estimates for these costs to the special case
\[
W(\mu,\nu)  :=\min_{\lambda\in\Pi(\mu,\nu)} \int_{\RH\times\RH} \min\{|x-y|,1\}\,d\lambda(x,y),
\]
which is used as a reference cost in our considerations. The key comparison between transportation costs $C_c(\mu,\nu)$ and $W(\mu,\nu)$ is given by the following Lemma, which is a generalization of  \cite[Lemma 5]{Seis} .

\begin{lemma}\label{comparison}
Let $\mu,\nu\in\M_+(\RH)$ with $\mu(\RH)=\nu(\RH)$ and let $c:[0,\infty]\to[0,\infty]$ be continuous and strictly increasing. Given any $\varepsilon>0$ we have the upper bound
\[
W(\mu,\nu) \le c^{-1}\left(\frac{C_c(\mu,\nu)}{\varepsilon}\right) \,\mu(\RH) +\varepsilon  + \frac{C_c(\mu,\nu)}{c(1)},
\]
if $\frac{C_c(\mu,\nu)}{\varepsilon}\in c\big([0,\infty]\big)$.
\end{lemma}
\begin{proof}
Define the following sets in $\RH\times\RH$:
\begin{align*}
K_1 & = \left\{|x-y|\le 1: c(|x-y|)\le \frac{C_c(\mu,\nu)}{\varepsilon} \right\}, \\
K_2 & = \left\{|x-y|\le 1: c(|x-y|)> \frac{C_c(\mu,\nu)}{\varepsilon} \right\}, \\
\textrm{and}\quad K_3 & = \big\{|x-y|> 1 \big\}. 
\end{align*}
Let $\lambda$ be an optimal plan for $C_c(\mu,\nu)$. Using $\lambda$ as a test plan for $W(\mu,\nu)$ we see that
\begin{align*}
W(\mu,\nu) & \le \int_{\RH\times\RH} \min\{|x-y|,1\}\,d\lambda(x,y) \\
& = \int_{K_1}\min\{|x-y|,1\}\,d\lambda(x,y) +\int_{K_2}\min\{|x-y|,1\}\,d\lambda(x,y) \\
& \qquad +\int_{K_3}\min\{|x-y|,1\}\,d\lambda(x,y) \\
& \le c^{-1}\left(\frac{C_c(\mu,\nu)}{\varepsilon}\right) \,\lambda(K_1) + \lambda(K_2) +\lambda(K_3).
\end{align*}
Here it is easy to see that
\[
\lambda(K_2)= \int_{K_2} \frac{c(|x-y|)}{c(|x-y|)} \,d\lambda(x,y) \le \frac{\varepsilon}{C_c(\mu,\nu)}\int_{K_2}c(|x-y|) \,d\lambda(x,y) \le \varepsilon
\]
and similarly
\[
\lambda(K_3) = \int_{K_3} \frac{c(|x-y|)}{c(|x-y|)} \,d\lambda(x,y) \le \frac{C_c(\mu,\nu)}{c(1)},
\]
which concludes the proof.
\end{proof}

\noindent
If $c$ satisfies the conditions \eqref{cost}, then $c(|x-y|)$ gives a metric in $\RH$. For this kind of cost functions the optimal transport problem admits a dual formulation due to the celebrated Kantorovich Duality Theorem (see  e.g. \cite[Theorem 5.10]{Villani2}).

\begin{theorem}\label{duality}
Let $\mu,\nu\in\M_+(\RH)$ with $\mu(\RH)=\nu(\RH)>0$. Suppose $c:[0,\infty]\to[0,\infty)$ satisfies conditions \eqref{cost}. Then we have the duality:
\begin{align*}
& \min \left\{ \int_{\RH\times\RH} c(|x-y|)\,d\lambda(x,y) : \lambda\in\Pi(\mu,\nu) \right\} \\
= & \max \left\{ \int_{\RH} v\,d\mu - \int_{\RH} v\,d\nu : \sup_{x\neq y}\frac{|v(x)-v(y)|}{c(|x-y|)} \le 1, \; v(\Diamond)=0 \right\}.
\end{align*}
In addition, the maximizers of the dual problem, called \emph{Kantorovich potentials}, satisfy
\[
v(x)-v(y)=c(|x-y|) \quad\textrm{for every}\; (x,y)\in\supp\lambda,
\]
where $\lambda\in\Pi(\mu,\nu)$ is an optimal plan for $C_c(\mu,\nu)$.
\end{theorem}

Since it is important for us, we emphasize that in our case any test function $v$ for the dual problem is bounded and Lipschitz in the euclidean distance:
\begin{equation}\label{costmax}
\sup_{x\in\RH} |v(x)| \le c(\infty) 
\quad\textrm{and}\quad \sup_{x,y\in\RN} \frac{|v(x)-v(y)|}{|x-y|} \le \sup_{t,s\in[0,\infty)} \frac{|c(t)-c(s)|}{|t-s|},
\end{equation}
In addition, information about the gradient of the Kantorovich potentials can also be extracted from the Duality Theorem.

\begin{corollary}\label{gradientofpot}
Suppose $c:[0,\infty]\to[0,\infty)$ is $C^1$ and satisfies \eqref{cost}. Let $\mu,\nu\in\M_+(\RH)$ with $\mu(\RH)=\nu(\RH)$. Let $\lambda\in\Pi(\mu,\nu)$ be an optimal plan for $C_c(\mu,\nu)$ and let $v$ be a corresponding Kantorovich potential.

Then there exists a set $E\subset\RN$ such that $\Leb^n(\RN\setminus E)=0$ and for any $(x,y)\in\supp\lambda\cap(E\times E)$ such that $x\ne y$ we have
\[
\nabla v(x)= \nabla v(y) = c'(|x-y|)\frac{x-y}{|x-y|} 
\]
In addition, $\nabla v(x)=0$, if $x\in E$ and $(x,\Diamond)\in\supp\lambda$ or $(\Diamond,x)\in\supp\lambda$.
\end{corollary}

\begin{proof}
Since $v$ is Lipschitz in the euclidean distance, we can apply the Rademacher Theorem to find a set $E\subset \RN$ such that $\Leb^n(\RN\setminus E)=0$ and $v$ is differentiable in $E$.

Now, given $(x,y)\in\supp\lambda\cap (E\times E)$ we have by the Duality Theorem
\[
c(|x-y|)+v(y)=v(x)= \inf_{y'\in\RN}\{ c(|x-y'|)+v(y') \}.
\]
Thus, if $x\ne y$ we can differentiate the function $f(y')=c(|x-y'|)+v(y')$ at its minimum point $y\in E$, which gives
\[
c'(|x-y|)\frac{y-x}{|x-y|}+\nabla v(y)=0, \quad \textrm{i.e.}\quad \nabla v(y)=c'(|x-y|)\frac{x-y}{|x-y|}.
\]
Similarly, the Duality Theorem implies
\[
c(|x-y|)-v(x)=-v(y)= \inf_{x'\in\RN}\{ c(|x'-y|)-v(x') \},
\]
from which we obtain by differentiation
\[
\nabla v(x) =c'(|x-y|)\frac{x-y}{|x-y|}.
\]

If $x\in E$ and $(x,\Diamond)\in\supp\lambda$, then using the Duality Theorem we see that
\[
c(\infty)-v(x)=0= \inf_{x'\in\RN}\{ c(\infty)-v(x') \},
\]
and again differentiation gives $\nabla v(x)=0$. The case where $(\Diamond,x)\in\supp\lambda$ can be handled similarly.
\end{proof}

In the uniqueness part of the proof on Theorem \ref{mainTHM} we will deal with certain integrals involving the gradient of a Kantorovich potential. These integrals can be estimated easily using optimal transport theory:

\begin{corollary}\label{firsttermest}
Suppose $c:[0,\infty]\to[0,\infty]$ is $C^1$ and satisfies \eqref{cost}. Let $\mu,\nu\in\M_+(\RH)$ with $\mu(\RH)=\nu(\RN)$ and let $v:\RH\to\R$ be a Kantorovich potential for $C_c(\mu,\nu)$. Assume that $\mu$ and $\nu$ are  mutually  singular, $\mu,\nu<<\Leb^n$ and $b\in L^1(\mu)\cap L^1(\nu)$. Then
\begin{align*}
& \bigg| \int_{\RN}\langle b, \nabla v \rangle d(\mu-\nu) \bigg| \\
\le \quad & \min\{\mu(\RN),\nu(\RN)\} \sup_{x,y\in F} \left| c'(|x-y|) \big\langle b(x)-b(y), \frac{x-y}{|x-y|} \big\rangle \right|,
\end{align*}
where $F:=(\supp\mu\cup\supp\nu)\cap\RN$.
\end{corollary}

\begin{proof}
Let $\lambda\in\Pi(\mu,\nu)$ be an optimal plan for $C_c(\mu,\nu)$. Since $\mu$ and $\nu$ are singular, there exists a Borel set $A\subset\RH$ such that $\mu(A)=0$ and $\nu(\RN\setminus A)=0$. Then $\lambda$ is concentrated on the set $(\RN\setminus A) \times A$.

Define functions $f,g:\RH\to\R$ by
\begin{align*}
f(x) & = \left\{ \begin{array}{ll} \langle b(x), \nabla v(x) \rangle, & x\in E \\
0, & x\in\RH\setminus E \end{array}\right. \\
g(y) & = \left\{ \begin{array}{ll} -\langle b(y), \nabla v(y) \rangle, & y\in E \\
0, & y\in\RH\setminus E \end{array}\right.
\end{align*}
Then $f\in L^1(\mu)$ and $g\in L^1(\nu)$, which means that
\begin{equation}\label{tp}
\int_{\RH} f \,d\mu + \int_{\RH} g\,d\nu = \int_{\RH\times \RH} \big(f(x)+g(y)\big)\,d\lambda(x,y).
\end{equation}
Here 
\[
\int_{\RH} f \,d\mu + \int_{\RH} g\,d\nu = \int_{\RN}\langle b, \nabla v \rangle d\mu -\int_{\RN}\langle b, \nabla v \rangle d\nu,
\]
since $\mu,\nu<<\Leb^n$ and $\Leb^n(\RN\setminus E)=0$. Thus we only need to estimate the right side of \eqref{tp}. 

To do this, we notice that $\lambda$ is concentrated on the set 
\[
\Lambda:= \big( \supp\mu\cap(E\cup\{\Diamond\})\cap \RH\setminus A \big) \times \big( \supp\nu\cap (E\cup\{\Diamond\})\cap A \big).
\]
For any $(x,y)\in\Lambda\cap\supp\lambda$ we can further apply Corollary \ref{gradientofpot}. This gives
\begin{align*}
& \int_{\RH\times \RH} \big(f(x)+g(y)\big)\,d\lambda(x,y) \\
= \quad & \int_{\Lambda\cap\supp\lambda\cap (E\times E)} \big( \Langle b(x), \nabla v(x) \Rangle - \Langle b(y), \nabla v(y) \Rangle \big) \, d\lambda(x,y)  \\
& 
+ \int_{\Lambda\cap\supp\lambda\cap ((E\times \{\Diamond\}) \cup (\{\Diamond\}\times E))} \big( f(x)+g(y) \big) \, d\lambda(x,y) \\
= \quad & \int_{\Lambda\cap\supp\lambda\cap (E\times E)} \Langle b(x)-b(y), \nabla v(x) \Rangle \, d\lambda(x,y) \\
= \quad & \int_{\Lambda\cap\supp\lambda\cap (E\times E)} \Langle b(x)-b(y), c'(|x-y|)\,\frac{x-y}{|x-y|} \Rangle d\lambda(x,y).
\end{align*}
Now the proof can be completed by taking the absolute value inside the integral and noticing that $\supp\lambda\subset F\times F$ and 
\[
\lambda(\Lambda\cap\supp\lambda\cap (E\times E))\le \min\{ \mu(\RN),\nu(\RN) \}.
\]
\end{proof}

Finally, we want to know that the total transport cost behaves well under convergence of measures.

\begin{lemma}\label{otandweakconvergence}
Suppose $c:[0,\infty]\to[0,\infty]$ satisfies conditions \eqref{cost}. Let $\rho_k,\rho\in\M(\RH)$, $k\in\N$, be measures such that $\rho_k(\RH)=\rho(\RH)=0$ for all $k\in\N$. Assume that $\rho_k\to\rho$ weakly as $k\to\infty$, that is
\[
\int_{\RH} \varphi \,d\rho_k \to \int_{\RH} \varphi \,d\rho \quad\textrm{for every}\; \varphi\in C^0(\RH). 
\] 
Then 
\[
C_c(\rho^+,\rho^-) \le \liminf_{k\to\infty} C_c(\rho_k^+,\rho_k^-).
\]
\end{lemma}

\begin{proof}
Let $v:\RH\to\R$ be such that $v(\Diamond)=0$ and $|v(x)-v(y)|\le c(|x-y|)$ for all $x,y\in\RH$. Then $v$ is continuous and bounded, so it follows from weak convergence and the Duality Theorem that
\[
\int_{\RH} v \,d\rho = \lim_{k\to\infty} \int_{\RH} v \,d\rho_k  \le \liminf_{k\to\infty} C_c(\rho_k^+,\rho_k^-).
\]
Taking the supremum over $v$ we get the claim.
\end{proof}


\section{Proof of  Proposition \ref{mainPROP} } \label{proofmainPROP}

\noindent
In this section we  prove the following more general version of Proposition \ref{mainPROP}.  

\begin{proposition}\label{moregeneralPROP}
Let $b:[0,T]\times\RN\to\RN$ be a vector field. Assume that there exists a continuous and nondecreasing function $G:[0,\infty)\to(0,\infty)$ satisfying $\int_r^\infty \frac{ds}{G(s)}=\infty$ for some $r>0$ such that
\[
\sup_{x\in\RN} \frac{|b(t,x)|}{G(|x|)} \in L^\infty(0,T) \tag*{(B1)}.  
\]
Assume also that there exists a continuous and nondecreasing function $\omega:[0,\infty)\to[0,\infty)$ satisfying $\int_0^r \frac{ds}{\omega(s)} = \infty$ for some $r>0$ such that
\[\label{b2s}
C_\omega:= \esup_{t\in(0,T)} \sup_{x, y\in\RN} \frac{|b(t,x)-b(t,y)|}{\omega(|x-y|)} <\infty. \tag*{(B2')}
\]
If $u\in L^1(0,T; L^1(\RN))$ is a weak solution of the Cauchy problem
\begin{equation}\label{CP}
\left\{ 
\begin{array}{rll}
\di_t u +\diver(b \,u) & =0 & \textrm{in}\; (0,T)\times\RN \\
u(0,\cdot) & =0 & \textrm{in}\; \RN,
\end{array}\right.
\end{equation}
then $u(t,\cdot)=0$ as $L^1(\R^n)$ functions, for each $t\in [0,T]$.  
\end{proposition}

\noindent
In order to prove Proposition \ref{moregeneralPROP}, we are given a function $u\in L^1(0,T;L^1(\RN))$, weak solution of the Cauchy problem \eqref{CP}. Associated to this solution $u$,  we define 
$$\aligned
\mu_k(t)&:= \chi_k u^+(t,\cdot)\Leb^n+m_k^+(t)\delta_\Diamond \\
\nu_k(t)&:= \chi_k u^-(t,\cdot)\Leb^n+m_t^-(t)\delta_\Diamond,\endaligned$$
where $m_k^\pm (t):= \left(-\int_{\RN} \chi_k(x) u(t,x)\,dx\right)^\pm $ , where the superscripts denote the positive and negative parts of a number, and $\chi_k$ are as in \eqref{cutoff} . Then $\mu_k(t)(\RH)=\nu_k(t)(\RH)$, and so we can consider the optimal transport problem
\[
D_{\delta,k}(t):= \inf_{\lambda\in\Pi(\mu_k(t),\nu_k(t))} \int_{\RH\times \RH} c_\delta(|x-y|)\,d\lambda(x,y),
\]
where the cost function is $c_\delta:[0,\infty]\to[0,\infty]$,
\[
c_\delta(r) = \int_0^r \frac{ds}{\omega(s)+\delta}.
\]
Since we can increase $\omega(s)$ when $s>0$ is large without affecting condition \ref{b2s}, we may assume that $\int_1^\infty \frac{ds}{\omega(s)} <\infty$. This ensures that $c_\delta$ is $C^1$, strictly increasing and satisfies conditions \eqref{cost}, so that all the results from the previous section are at our disposal.  Our strategy is to estimate the total cost $D_{\delta,k}$ and then use Lemma \ref{comparison} as $k\to \infty$ and $\delta\to 0$. 

\begin{lemma}\label{costcontinuity}
 If $k$ and $\delta$ are fixed, then $D_{\delta,k}(t)\to 0$  as $t\to 0$.
\end{lemma}

\begin{proof}
By the Duality Theorem it is sufficient to prove that as $t\to 0$
\begin{equation}\label{tto0}
D_{\delta,k}(t)=\sup_{v\in \mathcal{F}} \int_{\RN} v(x)\chi_k(x)u(t,x)\,dx \to 0, 
\end{equation}
where 
\[
\mathcal{F}:=\{v\in C(\RH): \sup_{x\ne y}\frac{|v(x)-v(y)|}{c_\delta (|x-y|)} \le 1, \;\; v(\Diamond)=0\}.
\]
Since $u\in L^1(0,T;L^1(\RN))$ we can use Lipschitz functions in the definition of solutions of \eqref{CP}. In particular, we can follow the arguments in the proof of Lemma \ref{wcont} in the case where $\varphi=v\chi_k$, $v\in\mathcal{F}$. Thus we get the following version of \eqref{t0eq}:
\[
-\int_{\RN} v(x)\chi_k(x) u(t_0,x) \,dx +\int_0^{t_0} \int_{\RN} \Langle b(t,x) , \nabla(v\chi_k)(x) \Rangle u(t,x) \,dx\,dt  = 0.
\]
Using the fact that $|\nabla(v\chi_k)|\le \frac{1}{\delta}+\frac{2c_\delta(\infty)}{G(k)}$ we arrive at the estimate
\begin{align*}
\Big| \int_{\RN} v(x)\chi_k(x)u(t_0,x)\,dx \Big| \le C_b \Big( \frac{1}{\delta}+\frac{2c_\delta(\infty)}{G(k)} \Big) \int_0^{t_0} \int_{\RN} |u(t,x)|\,dx\,dt,
\end{align*}
where $C_b:=\|b\|_{L^\infty(0,T;L^\infty(B(0,R_k))}<\infty$. Since the bound on the right hand side is independent of $v\in\mathcal{F}$ and goes to zero as $t_0\to 0$ , we obtain \eqref{tto0}.
\end{proof}

\noindent
When we study the weak derivative of $D_{\delta,k}$ the Kantorovich potentials come into play. To that end we denote by $v_k(t,\cdot)$ the Kantorovich potential for $D_{\delta,k}(t)$.

\begin{lemma}\label{wdiff}
The total cost $D_{\delta,k}$ has a weak derivative
\begin{align*}
\di_t D_{\delta,k}(t) & = \int_{\RN} \big\langle b(t,x), \nabla v_k(t,x) \big\rangle\, \chi_k(x) u(t,x) \,dx \\
& \qquad + \int_{\RN} v_k(t,x) \,\big\langle b(t,x), \nabla \chi(x) \big\rangle\, u(t,x) \,dx.
\end{align*}
\end{lemma}

\begin{proof}
Our possible derivative is clearly in $L^1(0,T)$, so integrability is not an issue. Using the duality and the assumption that $v_k(t,\Diamond)=0$ we get for any $h\in\R$
\begin{equation}\label{difference}
\begin{split}
& \quad D_{\delta,k}(t)-D_{\delta,k}(t-h) \\
\le &  \int_{\RN} v_k(t,x)\chi_k(x)u(t,x)\,dx -\int_{\RN} v_k(t,x)\chi_k(x)u(t-h,x)\,dx.
\end{split}
\end{equation}
Let $\psi\in C^\infty_c(0,T)$, $\psi\ge 0$, be a test function and $h\in\R$ so small that $\supp(\psi)\subset(|h|,T-|h|)$. It follows from a change of variables that
\[
\int_0^T \psi(t)(D_{\delta,k}(t)-D_{\delta,k}(t-h))\,dt= \int_0^T (\psi(t)-\psi(t+h))D_{\delta,k}(t)\,dt
\]
In particular we see that
\begin{equation}\label{leftside}
\lim_{h\to 0} \frac{1}{h}\int_0^T \psi(t)(D_{\delta,k}(t)-D_{\delta,k}(t-h))\,dt = -\int_0^T \psi'(t)D_{\delta,k}(t)\,dt.
\end{equation}

Similarly it follows from a change of variables that
\begin{align*}
& \int_0^T\int_{\RN} \psi(t)v_k(t,x)\,\chi_k(x)(u(t,x)-u(t-h,x))\,dx\,dt \\
= & \int_0^T\int_{\RN} (\psi(t)v_k(t,x)-\psi(t+h)v_k(t+h,x))\chi_k(x) u(t,x)\,dx\,dt \\
= & \int_0^T\int_{\RN} \di_t \Big( \int_{t+h}^t \psi(s)v_k(s,x)\,ds \Big) \chi_k(x) u(t,x)\,dx\,dt.
\end{align*}
If we have a smooth function $v_\varepsilon$ instead of $v_k$ above, we can use the fact that $u$ is a solution to the continuity equation to get
\begin{equation}\label{useofCP}
\begin{split}
& \int_0^T\int_{\RN} \di_t \Big( \int_{t+h}^t \psi(s)v_\varepsilon(s,x)\,ds \Big) \chi_k(x) u(t,x)\,dx\,dt \\
= & -\int_0^T\int_{\RN} \langle b(t,x), \int_{t+h}^t\psi(s)\nabla v_\varepsilon(s,x)\,ds \rangle \chi_k(x) u(t,x)\,dx\,dt \\
& \quad -\int_0^T\int_{\RN} \int_{t+h}^t \psi(s) v_\varepsilon(s,x)\,ds \langle b(t,x),\nabla\chi_k(x) \rangle u(t,x)\,dx\,dt.
\end{split}
\end{equation}
Since $v_k$ is Lipschitz in the euclidean distance, we can approximate it with smooth functions $v_\varepsilon$ so that
\begin{align*}
& \lim_{\varepsilon\to 0} \int_0^T\int_{\RN} \psi(t)v_\varepsilon(t,x)\,\chi_k(x)(u(t,x)-u(t-h,x))\,dx\,dt \\
= \quad & \int_0^T\int_{\RN} \psi(t)v_k(t,x)\,\chi_k(x)(u(t,x)-u(t-h,x))\,dx\,dt 
\end{align*}
and
\begin{align*}
 & \lim_{\varepsilon\to 0} \bigg( \int_0^T\int_{\RN} \langle b(t,x), \int_{t+h}^t\psi(s)\nabla v_\varepsilon(s,x)\,ds \rangle \chi_k(x) u(t,x)\,dx\,dt \\
 & \quad + \int_0^T\int_{\RN} \int_{t+h}^t \psi(s) v_\varepsilon(s,x)\,ds \langle b(t,x),\nabla\chi_k(x) \rangle u(t,x)\,dx\,dt \bigg) \\
= \quad & \int_0^T\int_{\RN} \langle b(t,x), \int_{t+h}^t\psi(s)\nabla v_k(s,x)\,ds \rangle \chi_k(x) u(t,x)\,dx\,dt \\
 & + \int_0^T\int_{\RN} \int_{t+h}^t \psi(s) v_k(s,x)\,ds \langle b(t,x),\nabla\chi_k(x) \rangle u(t,x)\,dx\,dt.
\end{align*}
Thus we get 
\begin{align*}
& \int_0^T\int_{\RN} \psi(t)v_k(t,x)\,\chi_k(x)(u(t,x)-u(t-h,x))\,dx\,dt \\
= & -\int_0^T\int_{\RN} \langle b(t,x), \int_{t+h}^t\psi(s)\nabla v_k(s,x)\,ds \rangle \chi_k(x) u(t,x)\,dx\,dt \\
 & -\int_0^T\int_{\RN} \int_{t+h}^t \psi(s) v_k(s,x)\,ds \langle b(t,x),\nabla\chi_k(x) \rangle u(t,x)\,dx\,dt,
\end{align*}
and furthermore, if we divide by $h$ and let $h\to 0$ we obtain
\begin{equation}\label{rightside}
\begin{split}
& \lim_{h\to 0} \frac{1}{h} \int_0^T\int_{\RN} \psi(t)v_k(t,x)\,\chi_k(x)(u(t,x)-u(t-h,x))\,dx\,dt \\
= & \int_0^T\int_{\RN} \psi(t) \langle b(t,x), \nabla v_k(t,x) \rangle \chi_k(x) u(t,x) \,dx\,dt \\
& + \int_0^T\int_{\RN} \psi(t) v_k(t,x) \langle b(t,x), \nabla \chi_k(x)\rangle u(t,x)\, dx\,dt.
\end{split}
\end{equation}
Now we want to combine this with \eqref{difference} and \eqref{leftside}. The limits \eqref{leftside} and \eqref{rightside} do not depend on the sign of $h$, but the inequality \eqref{difference} changes depending on whether we divide the sides by positive or negative $h$.
%
Thus by combining \eqref{difference}, \eqref{leftside} and \eqref{rightside} we have in fact proved that
\begin{equation}\label{weakdiff}
\begin{split}
& -\int_0^T \psi'(t)D_{\delta,k}(t)\,dt \\
= & \int_0^T\int_{\RN} \psi(t) \langle b(t,x), \nabla v_k(t,x) \rangle \chi_k(x) u(t,x) \,dx\,dt \\
& + \int_0^T\int_{\RN} \psi(t) v_k(t,x) \langle b(t,x), \nabla \chi_k(x)\rangle u(t,x)\, dx\,dt.
\end{split}
\end{equation}
To conclude the proof we need to show \eqref{weakdiff} for general test functions $\psi\in C^\infty_c(0,T)$. This is done by approximating $\psi^+$ and $\psi^-$ by smooth functions for which \eqref{weakdiff} holds.
\end{proof}

\noindent
Now we are ready to prove Proposition \ref{moregeneralPROP}. Note that Proposition \ref{mainPROP} immediately follows by the linearity of the equation, since (A1) and (A2) are particular cases of \ref{b1} (with $G(t)=1+t$) and \ref{b2s} (with $\omega(s)=s\,\log(e+\frac1s)$).

\begin{proof}[Proof of Proposition \ref{moregeneralPROP}]
Fix $\delta>0$ and $k\in\N$. We note that using Lemmas \ref{costcontinuity} and \ref{wdiff}  we see that
\begin{equation}\label{costestimate}
\begin{split}
D_{\delta,k}(t)  \leq & \Big| \int_0^T \int_{\RN} \Langle b(t,x), \nabla v_k(t,x) \Rangle \chi_k(x)u(t,x) \,dx\,dt \Big| \\
& + \Big|\int_0^T \int_{\RN} v_k(t,x) \,\big\langle b(t,x), \nabla\chi_k(x) \big\rangle\, u(t,x) \,dx\,dt \Big|.
\end{split}
\end{equation}
The first term on the left hand side can be estimated by Corollary \ref{firsttermest}, where we take into account that 
\[
c_\delta'(|x-y|)=\frac{1}{\omega(|x-y|)+\delta} \le \frac{1}{\omega(|x-y|)}.
\]
This gives
\[
\Big| \int_0^T \int_{\RN} \Langle b(t,x), \nabla v_k(t,x) \Rangle \chi_k(x)u(t,x) \,dx\,dt \Big| \le C_\omega \int_0^T \|u(t,\cdot)\|_{L^1(\RN)} \,dt,
\]
where $C_\omega<\infty$ by assumption \ref{b2s}. For the second term we only need the bounds on $v_k$ and $\nabla\chi_k$:
\begin{align*}
& \Big|\int_0^T \int_{\RN} v_k(t,x) \,\big\langle b(t,x), \nabla\chi_k(x) \big\rangle\, u(t,x) \,dx\,dt \Big| \\
 \le \quad & 2 C_G \, c_\delta(\infty)  \int_0^T \int_{\RN\setminus B(0,k)} |u(t,x)|\,dx\,dt,
\end{align*}
where
\[
C_G:=\esup_{t\in(0,T)} \sup_{x\in\RN} \frac{|b(t,x)|}{G(|x|)}<\infty
\]
by assumption \ref{b1}. Thus we obtain $D_{\delta,k}(t) \le C_{\delta,k}$, where
\[
C_{\delta,k}:=C_\omega \int_0^T \|u(t,\cdot)\|_{L^1(\RN)}\,dt + 2C_G \, c_\delta(\infty) \int_0^T \int_{\RN\setminus B(0,k)} |u(t,x)|\,dx\,dt.
\]
Now we denote $\mu(t):=u^+(t,\cdot)\Leb^n$ and $\nu(t):=u^-(t,\cdot)\Leb^n$. Since $\int_{\RN} u(t,x) \,dx=0$ by Lemma \ref{0int}, we see that $\mu_k(t)-\nu_k(t) \to \mu(t)-\nu(t)$ weakly. This enables us to apply first Lemma \ref{otandweakconvergence} and then Lemma \ref{comparison}, which gives for any $\varepsilon>0$ and $\delta>0$ small enough
\begin{equation}\label{costestimate2}
\begin{split}
W(\mu(t),\nu(t)) &  \le \liminf_{k\to\infty} W(\mu_k(t),\nu_k(t)) \\
& \le \limsup_{k\to\infty} c_\delta^{-1} \Big(\frac{C_{\delta,k}}{\varepsilon}\Big)\|u(t,\cdot)\|_{L^1(\RN)}+\varepsilon+ \frac{C_{\delta,k}}{c_\delta(1)} \\
& = c_\delta^{-1}\Big(\frac{C}{\varepsilon}\Big)\|u(t,\cdot)\|_{L^1(\RN)}+\varepsilon+ \frac{C}{c_\delta(1)},
\end{split}
\end{equation}
where 
\[
C:= \lim_{k\to\infty} C_{\delta,k} = C_\omega \int_0^T \|u(t,\cdot)\|_{L^1(\RN)}\,dt<\infty.
\]
As $\delta\to 0$, our assumption on $\omega$ implies that $c_\delta^{-1}(\frac{C}{\varepsilon}) \to 0$ and $c_\delta(1)\to \infty$. Thus it follows from \eqref{costestimate2} that $W(\mu(t),\nu(t)) \le \varepsilon$. Since this holds for any $\varepsilon>0$, we see that $W(\mu(t),\nu(t))=0$, which implies $\mu(t)=\nu(t)$. In other words, $u=0$.
%
\end{proof}

\section{Proof of  Theorem \ref{mainTHM} }\label{proofmainTHM}

\noindent
Let $\rho\in L^1(0,T;\M(\RN))$ be a solution of \eqref{0eq}. We want to repeat the arguments in the previous section, but we run into trouble if we try to integrate the gradient of a Lipschitz function over general measure $\rho_t$. So we  regularize the measures $\rho_t$, and take into account the fact that the regularized measures might not be solutions of the same equation.

\begin{lemma}\label{regularizedCP}
Let $\rho\in L^1(0,T;\M(\RN))$ be a solution of \eqref{0eq} and $\eta \in C_c^\infty(\RN)$. Define $\rho_t^\eta:=\eta \ast \rho_t$. Then for every $\varphi\in C_c^\infty([0,T)\times\RN)$ we have
\begin{equation}\label{regCP}
\begin{split}
& \int_0^T \int_{\RN} \big[ \di_t\varphi(t,x)+\Langle b(t,x), \nabla \varphi(t,x) \Rangle \big] \rho_t^\eta(x)\,dx\, dt \\ 
= & \int_0^T\int_{R^n}\int_{R^n} \Langle (b(t,y)-b(t,x)), \nabla\varphi(t,y) \Rangle \eta(x-y) \,dy \,d\rho_t(x)\,dt.
\end{split}
\end{equation}
\end{lemma}

\begin{proof}
Let $\varphi\in C_c^\infty([0,T)\times\RN)$. Then we use the definition of $\rho_t^\eta$ and Fubini's Theorem to calculate
\begin{equation}\label{convolution}
\begin{split}
& \int_0^T \int_{\RN} \big[ \di_t\varphi(t,x)+\Langle b(t,x), \nabla \varphi(t,x) \Rangle \big] \rho_t^\eta(x)\,dx\, dt \\ 
= & \int_0^T \int_{\RN} \big[ \di_t\varphi(t,x)+\langle b(t,x), \nabla \varphi(t,x) \rangle \big] \int_{\RN} \eta(x-y)\,d\rho_t(y)\,dx\,dt \\
= & \int_0^T \int_{\RN} \int_{\RN} \di_t\varphi(t,x)\eta(x-y)\,dx\,d\rho_t(y)\,dt \\
& \quad + \int_0^T \int_{\RN} \int_{\RN} \Langle b(t,x), \nabla\varphi(t,x) \Rangle \eta(x-y) dx\,d\rho_t(y)\,dt.
\end{split}
\end{equation}
Now, if we define $\psi(t,y):=\int_{\RN} \varphi(t,x)\eta(x-y)\,dx$, then $\psi\in C_c^\infty((0,T)\times\RN)$,
\begin{align*}
& \di_t\psi(t,y)=\int_{\RN} \di_t\varphi(t,x)\eta(x-y)\,dx \\
\textrm{and}\quad & \nabla\psi(t,y)= -\int_{\RN}\varphi(t,x)\nabla\eta(x-y)\,dx = \int_{\RN}\nabla\varphi(t,x)\eta(x-y)\,dx,
\end{align*}
where the last equality follows from integration by parts. Thus we may use the fact that $\rho$ solves \eqref{0eq} to get
\begin{align*}
& \int_0^T \int_{\RN} \int_{\RN} \di_t\varphi(t,x)\eta(x-y)\,dx\,d\rho_t(y)\,dt \\
= & \int_0^T \int_{\RN} \di_t\psi(t,y)\,d\rho_t(y)\,dt \\
= & -\int_0^T \int_{\RN} \langle b(t,y),\nabla\psi(t,y)\rangle d\rho_t(y)\,dt \\
= & -\int_0^T \int_{\RN}\int_{\RN} \Langle b(t,y), \nabla\varphi(t,x) \Rangle \eta(x-y)\,dx\,d\rho_t(y)\,dt.
\end{align*}
Combining this with \eqref{convolution} gives the claim.
\end{proof}

\noindent
Now we can proceed similarly to the previous section. For  fixed $\beta,\delta>0$ we define $c_{\beta,\delta}:[0,\infty]\to[0,\infty]$ by 
\[
c_{\beta,\delta}(r):= \beta \int_0^r \frac{ds}{\omega(s)+\delta}. 
\]
Then $c_{\beta,\delta}$ is $C^1$, strictly increasing and satisfies conditions \eqref{cost}, if we assume that $\int_r^\infty \frac{ds}{\omega(s)}<\infty$. As before, we may assume this, because modifying $\omega$ for large values does not affect the condition \ref{b2}. For $0<\alpha<1$ we choose a mollifier $\eta_\alpha\in C_c^\infty(\RN)$ such that $\eta_\alpha \ge 0$, $\int_{\RN} \eta_\alpha(x)\,dx=1$ and $\supp\eta_\alpha\subset B(0,\alpha)$. Let 
$$
\aligned
\mu_{\alpha,k}(t)&:=\chi_k (\eta_\alpha \ast \rho_t)^+ \Leb^n +m_{\alpha,k}^+(t) \\
 \nu_{\alpha,k}(t)&:=\chi_k (\eta_\alpha \ast \rho_t)^- \Leb^n +m_{\alpha,k}^-(t),
\endaligned$$
where $m_{\alpha,k}^\pm (t):=\left(-\int_{\RN} \chi_k(x)(\eta_\alpha \ast\rho_t)(x)\,dx\right)^\pm $, and $\chi_k$ as in \eqref{cutoff}. Then $\mu_{\alpha,k}(\RH)=\nu_{\alpha,k}(\RH)$. We denote
\[
D_{\alpha,\beta,\delta,k}(t) := \min \bigg\{ \int_{\RH\times\RH} c_{\beta,\delta}(|x-y|)\,d\lambda(x,y): \lambda\in\Pi(\mu_t,\nu_t) \bigg\},
\]
and let $v_{\alpha,\beta,\delta,k}(t,\cdot)$ be a Kantorovich potential for $D_{\alpha,\beta,\delta,k}(t)$. As before, we try to obtain estimates for $D_{\alpha,\beta,\delta,k}$ and then apply these estimates to our reference cost via Lemma \ref{comparison}. But first, we need a counterpart to Lemma \ref{costcontinuity}:

\begin{lemma}\label{costcontinuity2}
$D_{\alpha,\beta,\delta,k}(t)\to 0$, as $t\to 0$.
\end{lemma}

\begin{proof}
Just repeat the proof of Lemma \ref{costcontinuity}, but instead of using the definition of a solution to \eqref{CP} use Lemma \ref{regularizedCP}. This gives an extra term, which fortunately causes no problems.
\end{proof}

\noindent
Combining the proof of Lemma \ref{wdiff} with Lemma \ref{regularizedCP} we get the following result.

\begin{lemma}\label{wdiff2}
The cost $D_{\alpha,\beta,\delta,k}$ has a weak derivative:
\begin{align*}
\di_t D_{\alpha,\beta,\delta,k}(t) & = \int_{\RN} \big\langle b(t,x), \nabla v_{\alpha,\beta,\delta,k}(t,x) \big\rangle\, \chi_k(x) (\eta_\alpha \ast \rho_t)(x) \,dx \\
& \qquad + \int_{\RN} v_{\alpha,\beta,\delta,k}(t,x) \,\big\langle b(t,x), \nabla \chi(x) \big\rangle\, (\eta_\alpha\ast\rho_t)(x) \,dx \\
& \qquad + \int_{\RN} \int_{\RN} \Langle b(t,y)-b(t,x), \nabla \varphi_{\alpha,\beta,\delta,k}(t,x) \Rangle \eta_\alpha(y-x) \,dy\, d\rho_t(x)\,dt,
\end{align*}
where $\varphi_{\alpha,\beta,\delta,k}(t,x)=v_{\alpha,\beta,\delta,k}(t,x)\chi_k(x)$.
\end{lemma}

\begin{proof}
Repeat the proof of Lemma \ref{wdiff}, but for equality \eqref{useofCP} use Lemma \ref{regularizedCP}, which gives the extra term.
\end{proof}

\noindent
Now it is easy to estimate the actual transport cost.

\begin{lemma}\label{costestimate}
We have the estimate
\begin{align*}
D_{\alpha,\beta,\delta,k}(t) \le & \beta C_{\omega,k} \int_0^T |\rho_t|(\RN)\,dt \\
&  + 2\beta C_G \int_0^\infty \frac{ds}{\omega(s)+\delta} \int_0^T |\rho_t|(\RN\setminus B(0,k-1))\,dt \\
& +C_{\omega,k} \, \omega(\alpha) \Big( \frac{\beta}{\delta}+ \frac{\beta}{G(k)}\int_0^\infty \frac{ds}{\omega(s)+\delta} \Big) \int_0^T|\rho_t|(\RN),
\end{align*}
where
\begin{align*}
C_{\omega,k} & :=\esup_{t\in(0,T)}\sup_{x,y\in B(0,R_k+1)} \frac{|b(t,x)-b(t,y)|}{\omega(|x-y|)} \\
\textrm{and}\quad C_G \; & :=\esup_{t\in(0,T)}\sup_{x\in\RN} \frac{|b(t,x)|}{G(|x|)}.
\end{align*}
\end{lemma}

\begin{proof}
Just integrate Lemma \ref{wdiff2} while taking Lemma \ref{costcontinuity2} into account. Lemma \ref{wdiff2} gives three terms, which we estimate. First, it follows from Corollary \ref{firsttermest} that
\begin{equation}\label{1st}
\begin{split}
& \Big| \int_0^T \int_{\RN} \big\langle b(t,x), \nabla v_{\alpha,\beta,\delta,k}(t,x) \big\rangle\, \chi_k(x) (\eta_\alpha \ast \rho_t)(x) \,dx \Big| \\
\le \quad & C_{\omega,k} \int_0^T |\rho_t|(\RN)\, dt.
\end{split}
\end{equation}
For the second term recall that 
\[
\big| \Langle b(t,x), \nabla\chi_k(x) \Rangle \big| \le |b(t,x)|\frac{2}{G(|x|)} \le 2C_G.
\]
Thus it is easy to see that
\begin{equation}\label{2nd}
\begin{split}
& \Big| \int_0^T \int_{\RN} v_{\alpha,\beta,\delta,k}(t,x) \,\big\langle b(t,x), \nabla \chi(x) \big\rangle\, (\eta_\alpha\ast\rho_t)(x) \,dx \Big| \\
\le \quad & 2C_G c_{\beta,\delta}(\infty) \int_0^T |\rho_t|(\RN\setminus B(0,k-1))\,dt.
\end{split}
\end{equation}
Finally, we note that for $\varphi_{\alpha,\beta,\delta,k}(t,x)=v_{\alpha,\beta,\delta,k}(t,x)\chi_k(x)$ we have
\[
|\nabla \varphi_{\alpha,\beta,\delta,k}| \le \frac{\beta}{\delta}+\frac{2 c_{\beta,\delta}}{G(k)},
\]
which implies
\begin{equation}\label{3rd}
\begin{split}
& \Big| \int_0^T \int_{\RN} \int_{\RN} \Langle b(t,y)-b(t,x), \nabla \varphi_{\alpha,\beta,\delta,k}(t,x) \Rangle \eta_\alpha(y-x) \,dy\, d\rho_t(x)\,dt \Big| \\
\le \quad & C_{\omega,k}\omega(\alpha)\Big( \frac{\beta}{\delta}+\frac{2 c_{\beta,\delta}(\infty)}{G(k)} \Big) \int_0^T |\rho_t|(\RN)\,dt.
\end{split}
\end{equation}
Combining the estimates \eqref{1st}, \eqref{2nd} and \eqref{3rd} and writing out $c_{\beta,\delta}(\infty)$ gives the claim.
\end{proof}

\noindent
Now we can prove the uniqueness part of Theorem \ref{mainTHM}.

\begin{proof}[Proof of uniqueness in Theorem \ref{mainTHM}]
It suffices to prove that if $\rho\in L^1(0,T;\M(\RN))$ solves \eqref{0eq}, then $\rho=0$. Fix $k\ge 2$, and choose parameters $\alpha$, $\beta$ and $\delta$ depending on $k$. For ease of notation we denote
\[
I:= \int_0^T |\rho_t|(\RN)\,dt \quad\textrm{and}\quad I_k:=\max \left\{\int_0^T |\rho_t|(\RN\setminus B(0,k-1))\,dt, \frac{1}{k} \right\}.
\]
Now we choose $\beta_k:=\frac{1}{C_{\omega,k}I+1}$, $\delta_k>0$ such that
\[
\int_0^\infty \frac{ds}{\omega(s)+\delta_k} = \frac{C_{\omega,k}I+1}{2(C_G+1)I_k}
\]
and $0<\alpha_k<1$ such that
\[
\omega(\alpha_k) \le \frac{1}{\Big( \frac{\beta_k}{\delta_k}+ \frac{\beta_k}{G(k)}\int_0^\infty \frac{ds}{\omega(s)+\delta_k} \Big) (C_{\omega,k}I+1)}.
\]
Now we set $\mu_k(t):=\mu_{\alpha_k,k}(t)$, $\nu_k(t):=\nu_{\alpha_k,k}(t)$ and $c_k:=c_{\beta_k,\delta_k}$. We consider the optimal transport problem $D_k(t):=D_{\alpha_k,\beta_k,\delta_k,k}(t)$. Using Lemma \ref{costestimate} we see that $D_k(t)\le 3$. Lemma \ref{comparison} then gives for any $\varepsilon>0$ the inequality
\begin{equation}\label{costest}
W(\mu_k(t),\nu_k(t)) \le c_k^{-1}(3/\varepsilon) +\varepsilon + \frac{3}{c_k(1)}.
\end{equation}
Since $\rho_t(\RN)=0$ by Lemma \ref{0int}, we see that $\mu_k(t)-\nu_k(t) \to \rho_t$ weakly. Thus we get from Lemma \ref{otandweakconvergence} 
\[
W(\rho_t^+,\rho_t^-) \le \liminf_{k\to\infty} W(\mu_k(t),\nu_k(t)).
\] 
Combining this with \eqref{costest} we see that to prove the Theorem it suffices to show that $c_k(r)\to\infty$ for any $r>0$ (which also implies that $c_k^{-1}(r)\to 0$ for any $r>0$). So let's do just that.\\
\\
Because of our choice of $\beta_k$ and $\delta_k$ we have
\[
c_k(\infty)=\beta_k\int_0^\infty \frac{ds}{\omega(s)+\delta_k} \to\infty, \quad\textrm{as}\; k\to\infty. 
\]
On the other hand, it holds that
\[
\beta_k\int_0^\infty \frac{ds}{\omega(s)+\delta_k} \le \beta_k\int_0^r \frac{ds}{\omega(s)+\delta_k} + \beta_k\int_r^\infty \frac{ds}{\omega(s)},
\]
where we can use the bound $\beta_k \le 1$ to see that
\[
\beta_k\int_r^\infty \frac{ds}{\omega(s)} \le \int_r^\infty \frac{ds}{\omega(s)} <\infty.
\]
Thus we obtain $c_k(r)=\beta_k\int_0^r \frac{ds}{\omega(s)+\delta_k}\to\infty$.
\end{proof}

\noindent
As mentioned in the introduction, the existence of solutions to the Cauchy problem for the continuity equation \eqref{CPb} in $L^1(0,T; L^1(\R^n))$ may fail in general. Indeed, given a vector field $b$ satisfying \ref{b1} and \ref{b2}, it admits a well defined flow map $\phi_t:\R^n\to\R^n$ solving
$$\begin{cases}\dot{\phi_t(x)}=b(t,\phi_t(x))\\\phi_0(x)=x\end{cases}$$
but it is far from clear wether $\phi_t$ will (or will not) preserve the set of absolutely continuous measures. In \cite{CJMO} (see also \cite{CJMO2}) this situation is analyzed in terms of $\diver(b)$. In particular, a smooth vector field $b$ is provided for which the flow does \emph{not} preserve Lebesgue measurable sets (both backward and forward in time). As shown in \cite{CJMO2}, and modulo other conditions, a certain sub-exponential degree of integrability for $\diver (b)$ suffices to guarantee the preservation of Lebesgue measurable sets. Below, we present some examples to show that the exponential integrability of the divergence $\diver(b)$ and the Osgood modulus of continuity  \ref{b2} are independent conditions. By simplicity we write the details in the plane.\\
\\
It is easy to find vector fields $b$ satisfying the Osgood continuity condition \ref{b2} but with $\diver(b)\notin L^{\infty}$. For instance, let $f_1, f_2:\R^2\to\R$ be $BMO$ functions with compact suport (in the next section we provide more information about the $BMO$ space). Then, define the vector field
$$
b(x):= (I_1*f_1(x),   I_1*f_2(x)),
$$
where $I_1(x)=c/ | x|$ is the Riesz potential of order $1$. Then $b$ belongs to the Zygmund class, and so satifies an Osgood condition. Clearly, $b$ has derivatives locally in $L^p$ for all $p<\infty$. Likewise, the divergence is
$$
\diver(b) = R_1*f_1 +   R_2*f_2 ,
$$
where $R_j$ denotes the $j$-th Riesz transform. Clearly, $\diver b\notin L^\infty$ in general, even though it is exponentially integrable because $\diver(b)\in BMO$. \\
\\
In order to get examples of vector fields $b$ satisfying \ref{b2} but with divergence not exponentially integrable we have to be more precise. Let $b$ a vector field with compact support, smooth in $\R^2\setminus \{(0,0)\}$ and such that in a convenient neighbourhood of the origin is defined by
$$
b(x,y)=(x f(x,y), y f(x,y) ),
$$
whith $f(x,y) =  \log(x^2+y^2)(\log(-\log(x^2+y^2)))$. When $\delta$ is small, the modulus of continuity $\omega(\delta)$ of $b$ is like $\delta \log(\delta^{-2}) \, \log(\log(\delta^{-2}))$ and so, $b$ satisfies the Osgood condition. Besides, in a neighbourhood of the origin,
$$
\diver (b) =2\,f(x,y) + \log(-\log(x^2+y^2)) +1)
$$
and then $\diver (b)$ is not exponentially integrable because $f$ does not admit any $BMO$ majorant. It is worth mentioning, though, that  $\diver (b)$ is subexponentially integrable.\\
\\
A slight modification gives an example of a vector field $b$ with exponentially integrable $\diver(b)$ and not satisfying the Osgood condition \ref{b2}. Consider the vector field $b$ with compact support, smooth in $\R^2\setminus \{(0,0)\}$ and such that in a convenient neighbourhood of the origin is defined by
$$
b(x,y)=(x g(x,y), -y g(x,y) ),
$$
whith $g(x,y) =  \log(x^2+y^2)(\log(-\log(x^2+y^2)))^2$. In this case, the modulus of continuity $\omega(\delta)$ of $b$ is as $\delta \log(\delta^{-2}) \, (\log(\log(\delta^{-2}))^2$ when $\delta$ is small, thus Osgood condition does not hold.
However, $\diver(b)$ is exponentially integrable because in a neighbourhood of the origin
$$
\diver (b) = 2\frac{x^2-y^2}{x^2+y^2} \left[ (\log(-\log(x^2+y^2)))^2 + 2  \log(-\log(x^2+y^2)) \right],
$$
that is, $|\diver (b)| \le C |\log(x^2+ y^2)|$ and so $\diver(b)$ has a $BMO$ majorant.

\section{  Proof of Theorem \ref{mainH1} } \label{applications}

\noindent
In this section, we prove Theorem \ref{mainH1}. In order to do so, we need to recall some fundamentals on quasiconformality. We say that a sense-preserving homeomorphism $\phi:\R^n\to \R^n$ is \emph{$K$-quasiconformal} if $\phi$ belongs to the local Sobolev space $W^{1,n}_{loc}(\R^n;\R^n)$ and its distributional differential $D\phi(x)$ satisfies the distortion inequality,
$$|D\phi(x)|^n\leq K\, J(x,\phi)$$
at almost every point $x\in\R^n$. Here $|D\phi(x)|$ denotes the operator norm of the differential matrix $D\phi(x)$, and $J(x,\phi)$ denotes the jacobian determinant $\det D\phi(x)$. The inverse of a $K$-quasiconformal map is a $K^{n-1}$-quasiconformal map. The composition of a $K_1$-quasiconformal map and a $K_2$-quasiconformal map is a $K_1K_2$-quasiconformal map. For more about this topic, we refer the interested reader to the monograph \cite{IM}.\\
\\
In classical harmonic analysis and Calder\'on-Zygmund theory, one introduces $H^1(\R^n)$ as the space of $L^1(\R^n)$ functions $h:\R^n\to \R$ such that all the Riesz transforms $R_jh =c_n\,P.V. \frac{x_j}{|x|^{n+1}}\ast h$ , $j=1,\dots,n$ belong to $L^1(\R^n)$ as well. One indeed can define
$$\|h\|_{H^1(\R^n)}:= \|h\|_{L^1(\R^n)}+\sum_{j=1}^n\|R_jh\|_{L^1(\R^n)}.$$
The Hardy space $H^1(\R^n)$ is the appropriate substitute in harmonic analysis for $L^1(\R^n)$, on which Calder\'on-Zygmund theory fails. It also appears as the topological predual of $BMO(\R^n)$, 
$$H^1(\R^n)^\ast = BMO(\R^n)$$
by Fefferman Theorem \cite{Feff}. Let us recall here that $BMO(\R^n)$ consists of functions $b:\R^n\to\R$ with \emph{bounded mean oscillation}, that is,
$$\|u\|_{BMO(\R^n)}=\sup_{B\subset \R^n}\inf_{c\in\R}\frac1{|B|}\int_B|u(x)-c| \,dx<\infty.$$
The dual pairing $\langle h, b\rangle$, $h\in H^1(\R^n)$ and $b\in BMO(\R^n)$, is given by an absolutely converging integral only in some particular cases, but not in general. It is also of interest for us the space $VMO(\R^n)$ of functions with \emph{vanishing mean oscillation}, which is the closure of 
$$C_c(\R^n)=\{f:\R^n\to\R; f\text{ is continuous and has compact support}\}$$
under the $\|\cdot\|_{BMO}$-topology. It is a proper, closed subspace of $BMO(\R^n)$, whose topological dual is precisely $H^1(\R^n)$, 
$$
VMO(\R^n)^\ast = H^1(\R^n).
$$
In Calder\'on-Zygmund theory, $BMO(\R^n)$ is the appropiate substitute for $L^\infty(\R^n)$, while $VMO(\R^n)$ is the one for $C_c(\R^n)$. The following theorem was proven by Reimann in \cite{Re2}.

\begin{theorem}[Reimann, \cite{Re2}]
Given $K\geq 1$ and an integer $n\geq 2$, there exists a constant $C(n, K)$ with the following property. If $\phi:\R^n\to\R^n$ is $K$-quasiconformal, and $b\in BMO(\R^n)$, then $b\circ\phi^{-1}\in BMO(\R^n)$. Moreover, one has
$$\frac{1}{C(n,K)}\,\|b\|_{BMO}\leq \|b\circ\phi^{-1}\|_{BMO}\leq C(n,K) \|b\|_{BMO}$$
for each $b\in BMO(\R^n)$. 
\end{theorem}

\noindent
Reimann's result is optimal. More precisely, if $\phi:\R^n\to\R^n$ is a sense-preserving homeomorphism for which $C_\phi(b)=b\circ\phi$ induces a topological isomorphism on $BMO(\R^n)$, then necessarily $\phi$ is $K$-quasiconformal, for some $K\geq1$ depending only on $n$ and $\|C_\phi\|_{BMO}$, see \cite{astmich}. If we want to dualize the above statement, then the composition operator $C_{\phi}$ must be replaced by the transport operator $T_\phi$,
$$\langle T_\phi h, g\rangle = \langle h, C_{\phi}g\rangle.$$
The above equality may not make sense in general, even in the cases when the dual pairing $\langle\cdot,\cdot\rangle$ is a true integral. However, quasiconformal maps are known to preserve Lebesgue null sets, and so the change of variables formula holds. As a consequence, if $h\in L^1_{loc}$ then the transport operator $T_\phi$ maps the absolutely continuous measure $d\mu=h(x)\,dx$ to another absolutely continuous measure $T_\phi(d\mu)= \phi _\sharp(d\mu)$ whose density comes from the change of variables formula, 
$$T_\phi (d\mu) = h(\phi^{-1}(x))\,J(x,\phi^{-1})\,dx.$$
By abbusing notation, one could simply write 
$$T_\phi(h)=\phi_\sharp h= h(\phi^{-1})\,J(\cdot, \phi^{-1}).$$
The following corollary looks at this fact from the $H^1(\R^n)$ perspective. 

\begin{corollary}\label{H1sufficient}
Let $\phi:\R^n\to\R^n$ be $K$-quasiconformal. Then the transport operator $T_\phi$ maps measures with $H^1(\R^n)$ density boundedly into themselves. That is, if $h\in H^1(\R^n)$ then $T_\phi  h = h(\phi^{-1})\,J(\cdot,\phi^{-1})$ is an $H^1(\R^n)$ function, and
$$\|T_\phi h\|_{H^1(\R^n)}\leq C(n,K)\,\|h\|_{H^1(\R^n)},$$
for each $h\in H^1(\R^n)$.  
\end{corollary}
\begin{proof}
We will prove that $T_\phi:H^1(\R^n)\to H^1(\R^n)$ is the adjoint of the composition operator $C_{\phi }:VMO(\R^n)\to VMO(\R^n)$, $C_{\phi }b = b\circ\phi $. The rest of the proof follows by Schauder's Theorem, 
$$
\|T_\phi\|_{H^1\to H^1}= \|C_{\phi }^\ast\|_{H^1\to H^1}\leq  \|C_{\phi } \|_{VMO\to VMO}\leq C(n,K)
$$ 
together with Reimann's previous result. Given $h\in H^1(\R^n)$, it suffices to see that if $h\in H^1(\R^n)$ and $g\in VMO(\R^n)$, then 
\begin{equation}\label{equality}
\langle T_\phi h, g\rangle = \langle h, C_{\phi }g\rangle
\end{equation}
where $\langle\cdot,\cdot\rangle$ denotes the $VMO-H^1$ duality. By density, we can assume that $g\in C_c(\R^n)$. But then \eqref{equality} becomes an equality of absolutely converging integrals,
$$\int_{\R^n} T_\phi h(x)\,g(x)\,dx=\int_{\R^n}h(y)\,g(\phi^{-1}(y))\,dy$$
which is easily seen to be true due to the change of variables formula. Thus, the claim follows.
\end{proof}

\noindent
We are now in position of proving Theorem \ref{mainH1}.

\begin{proof}[Proof of Theorem \ref{mainH1}]
According to \cite[Theorem 5]{CJMO3} (see also \cite{Re}), the Cauchy problem
$$
\begin{cases}
\dot\phi_t(x)=b(t,\phi_t(x))\\\phi_0(x)=x\end{cases}
$$ 
admits a unique flow $\phi_t:\R^n\to\R^n$ of $K_t$-quasiconformal maps, where
$$K_t\leq \exp\left((n-1)\int_0^t2\|S_A b(s,\cdot)\|_\infty\,ds\right),\hspace{1cm}0\leq t\leq T.$$
By Corollary \ref{H1sufficient}, we deduce that the transport operator $T_{\phi_t}$ maps measures with $H^1(\R^n)$ density boundedly into themselves. In particular, at every fixed time $t\in[0,T]$ we can define an $H^1(\R^n)$ function as follows,
$$h(t,\cdot)=T_{\phi_t}h_0,$$
or more explicitly $h(t,x)=h_0(\phi_t^{-1}(x))\,J(x,\phi_t^{-1})$, with norm
$$\|h(t,\cdot)\|_{H^1(\R^n)}\leq C(n,K_t)\,\|h_0\|_{H^1(\R^n)}.$$
In particular, $h\in L^\infty(0,T; H^1(\R^n))$, since $1\leq K_t\leq K_T$ for all $0\leq t\leq T$. It is an easy exercise to prove that $h(t,\cdot)$ solves the Cauchy problem \eqref{H1problem}. For uniqueness, we simply note that $b$ is in the assumptions of Theorem \ref{mainTHM}, which forces \eqref{H1problem} to have a unique measure valued solution. 
\end{proof}

\noindent
As in the $BMO$ setting, quasiconformality can also be characterized in $H^1$ terms.

\begin{corollary}\label{H1necessary}
Let $\phi:\R^n\to\R^n$ be a sense-preserving homeomorphism. Assume that the transport operator $T_\phi$ maps absolutely continuous measures with $H^1(\R^n)$ density boundedly into themselves. That is, there is a constant $C_0$ with the following property. If $d\mu=h\,dx$ with $h\in H^1(\R^n)$, then $T_\phi(d\mu)$ is absolutely continuous and 
$$\left\|\frac{d}{dx}\, T_\phi(d\mu)\right\|_{H^1(\R^n)}\leq C_0\,\|h\|_{H^1(\R^n)}.$$
Then $\phi$ must be $K$-quasiconformal, for some $K=K(C_0,n)\geq 1$.
\end{corollary}
\begin{proof}
By assumptin, given any $h\in H^1(\R^n)$, the measure $T_\phi(h\,dx)$ is absolutely continuous with $H^1(\R^n)$ density. Thus, for each $b\in BMO(\R^n)$ we can define a linear functional $(T_\phi b)^\ast :H^1(\R^n)\to\R$ as follows,
$$
\langle(T_\phi b)^\ast,h\rangle := \langle b,T_\phi h\rangle.
$$
By our assumption, $(T_\phi b)^\ast$ is well defined and bounded, with norm 
$$\|(T_\phi b)^\ast\|_{BMO(\R^n)}\leq C(n)\, C_0\,\|b\|_{BMO(\R^n)}$$
where $C(n)$ comes from Fefferman's duality. Therefore we can identify $(T_\phi b)^\ast$ with a $BMO(\R^n)$ function. On the other hand, if $h$ is smooth, compactly supported, and has zero integral, then $T_\phi h$ is a nice finite Borel measure with compact support, and the action $\langle b,T_\phi  h\rangle$ is an absolutely converging integral. By the definition of image measure one has
$$\langle b,T_\phi  h\rangle = \int b\, d(\phi_\sharp( h\,dx))=\int b(\phi(x))\,h(x)\,dx$$
whence
$$
\langle(T_\phi b)^\ast,h\rangle=\int b(\phi(x))\,h(x)\,dx. 
$$
It then follows that $(T_\phi b)^\ast=C_\phi b$, modulo additive constants. Using again that $T_\phi$ acts boundedly on $H^1(\R^n)$, and combining it with \cite[Theorem 3]{astmich}, the proof follows. 
\end{proof}

\noindent
As a consequence of Corollary \ref{H1necessary}, we deduce that if for a vector field $b$ the problem \eqref{H1problem} admits a unique weak solution in $H^1(\R^n)$, then either $b$ generates flows of quasiconformal maps, or it does not admit even a flow of sense-preserving homeomorphisms. A similar phenomena occurs with the transport equation in the $BMO(\R^n)$ setting (see \cite{CJMO3}).

\noindent
 A. Clop: albertcp@mat.uab.cat\\
H. Jylh\"a: hjjylh@gmail.com\\ 
J. Mateu: mateu@mat.uab.cat\\
J. Orobitg: orobitg@mat.uab.cat\\
\\
Departament de Matem\`atiques\\
Facultat de Ci\`encies, Campus de la U.A.B.\\
08193-Bellaterra (Barcelona)\\
Catalonia\\

\end{document}